\documentclass[11pt,letterpaper]{article}
\usepackage[margin=3cm]{geometry} 
\usepackage{amsmath, amsfonts, amssymb}
\newcommand*\rfrac[2]{{}^{#1}\!/_{#2}}
\newtheorem{prop}{Proposition}[section]
\newtheorem{theorem}{Theorem}[section]
\newtheorem{lemma}{Lemma}[section]
 \newenvironment{proof}[1][Proof]{\begin{trivlist}
\item[\hskip \labelsep {\bfseries #1}]}{\end{trivlist}}

\title{Poisson Regression with Survey Data}
\author{Seyed Jalil Kazemitabar}
\date{{}}
\begin{document}
\maketitle
\section*{Abstract}
We propose a way to remove the bias of a Poisson regression when the subjects are partially observed. In this paper we address this issue under certain assumptions about the missing-data generating process. We fix the total number of observed subjects and allow individual subjects to be observed randomly. This theme is relevant when a researcher is provided with a survey data not covering the whole population. A highlighting result is that if subjects are observed according to a random sampling without replacement, a Poisson distribution with sampling-ratio-adjusted mean is an asymptotically consistent model of the observed count variable. An innovative asymptotic regime is employed to derive the results.

\section{Motivation}
Suppose $N_j$ is the number of subjects for case $j$ with the feature vector $X_j$. Say $N_j$ is the size of population in city $j$. Let's assume the following Poisson DGP holds for the count variable $N_j$:
\begin{equation}
N_j \sim \text{Poisson}(X_j^\prime \beta),\quad j=1,\ldots,J
\end{equation}
The goal is to estimate coefficients, $\beta$. 

Suppose further that we are given a random sample from the entire set of subjects.  Denote $n_j$ as the number of subjects observed out of $N_j$ incidences. The state of the art in the literature would assume that each subject pertaining case $j$ is observed independently with probability $\gamma_j$, the conditional distribution of $n_j$ would be:
\begin{equation}
n_j \sim \text{Binomial}(N_j, \gamma_j),\quad j=1,\ldots,J
\end{equation}
The $J$ conditional pdf's in (2) will be independent, conditioning on $N_j,\; j=1,\ldots,J$. Solving for unconditional pdf of $n_j$'s would be somewhat straightforward. [citing some previous works here] We relax the assumption of $\gamma_j$ being constant and known and replace it with a random sampling assumption from the entire $\sum N_j$ subjects, so that $\sigma n_j$ is fixed. This assumption is the point of distinction between our paper and the previous works. We merely see the vector of $n_j$'s and the total sum of population, $\sum N_j$. One conclusion of this assumption is that the number of subjects observed from a case is proportional in expectation to the size of the case, conditional on the size of the case. Our assumption poses challenging conditions on the joint conditional pdf of observed count variables. First, $\gamma_j=\frac{N_j}{\sum N_j}$ is not known since real size of each case is not observable. Second, $n_j$'s are not independent anymore. The more subjects sampled from one case the less likely subjects from the other case are sampled out. We tackle this issues by introducing an appropriate asymptotic regime in the following sections.

This assumption is relevant because often there is no census of the entire population and only a restricted survey of subjects (e.g. people) is available. While there is a chance that a survey is stratified across cases (e.g cities), we go for a random sampling assumption. We will examine both sampling with replacement and sampling without replacement and leave the stratified sampling case to later works.

In section 2, we will describe the sampling assumptions and the ensuing conditional distribution of $n_j$'s. The asymptotic regime is mentioned in section 3. Sections 4 and 5 approximate asymptotic conditional pdf under the specified asymptotic regime. Section 6 will try to compute the unconditional pdf of observed count variables. The paper ends with a discussion of potential improvements and further works.

\section{Randomly Sampled (Observed) Count Variable}
\subsection{Random Sampling with Replacement}
Take $N_*$ and $n_*$ as the sum of $N_j$'s and $n_j$'s respectively and denote $\textbf{n}=(n_1,n_2^,\ldots,n_J)$ the vector of observed counts. Suppose we only observe a random sample with replacement of a fixed size $n_*$ from the entire set of $N_*$ subjects. In other terms, the following holds about the distribution of $n_j$'s:
\[
\textbf{n}\: |\: N_1,\ldots, N_J
 \sim \text{Multinom}\left(n_*,\; p_1,\ldots, p_J\right)
\]
where $p_j=\frac{N_j}{N_*},\; j=1,\ldots,J$. Random sampling with replacement (RW), though not a convincing assumption regarding a survey data of the population, is held as a starting point prior to analyzing the more appealing assumption introduced in the following section.

\subsection{Random Sampling without Replacement}
Keeping the same notations, assume instead that the subjects are observed through a random sampling procedure without replacement (RWO),
\[
\textbf{n}\: |\: N_1,\ldots, N_J \sim \text{Multi. Hypergeometric}(n_*, N_*, N_1,\ldots,N_J)
\]

Unconditional pdf of count variables is a complicated combinatorial expression. A more useful formula is the asymptotically consistent estimation of the joint pdf. We derive it in two steps, first, asymptotics of the conditional distribution and then the asymptotics of the unconditional pdf. 

\section{The Asymptotic Regime}
\subsection{Why conventional regime is not appropriate}
It is known that Multinomial and Multi Hypergeometric random vectors converge in distribution to Multivariate Normal, under the following asymptotic regime:
\begin{itemize}
\item [(P1)] $n_*\to\infty$
\item [(P2)] $n_*/N_*$ and $p_j=\frac{N_j}{N_*}$ for $j=1,\ldots,J$ are fixed
\end{itemize}
Indeed, asymptotic conditional distribution of $\textbf{n}$ would be:
\[
\begin{array}{c}
\textbf{n}\: |\: N_1,\ldots, N_J\overset{d}{\longrightarrow}\mathbb{N}(\mu,\Sigma) \\
\mu=(n_*\; p_1,\ldots,n_*\; p_J),\; \Sigma=[n_*\; p_i\; p_j]_{ij}
\end{array}
\]
(P1) and (P2) requires that $N_j$'s can be sufficiently large. This would in return impose an asymptotic assumption on the features, $X_j$'s, for all cases. This latter could be undesirable in some situations. Any reasonable distribution of features induce such small enough values that dis-validate any asymptotic approximation based on this assumption.

On the other hand, Poission pdf of $N_j$'s can not be combined with the conditional normal pdf of $\textbf{n}$ easily. [describe further]

\subsection{An alternative}
Our asymptotic regime impose no assumption on $X_j$'s. Instead of growing the size of all cases we increase the number of cases, in asymptotics. 

Since we are interested in asymptotic behavior of observed counts, we assume an infinite sequence of cases in which the actual counts of the subjects are $(N_1,N_2,N_3,\ldots)$. Throughout, this sequence is assumed nonrandom, or equivalently, our statements are conditioned on it. For a fixed $J$, we assume that the counts $\textbf{n}^{(J)}=(n_1^{(J)},n_2^{(J)},\ldots,n_J^{(J)})$ are sampled from $(N_1,\ldots,N_J)$ randomly with replacement. Let $n_*^{(J)}$ be the total number of subjects sampled at this stage, which we assume to be a known nonrandom quantity. Note that $n_*^{(J)}=\sum_{j=1}^{J}n_*^{(j)}$.

We are interested in the asymptotic distribution of $\textbf{n}^{(J)}$ as $J\to\infty$.
Let $N_*^{(J)}=\sum_{j=1}^{J}N_j$ be the total actual count in the first $J$ cases, and $p_j^{(J)}=N_j/N_*^{(J)}$ be the corresponding fraction. We set $\textbf{P}^{(J)}=(p_1^{(J)},\ldots,p_J^{(J)})$. We consider the following asymptotic regime:
\begin{itemize}
\item [(A1)] Infinite population size: $\sum_{j=1}^{\infty}N_j=\infty$, or equivalently, $N_*^{(J)}\to\infty$
\item [(A2)] Infinite sample size: $n_*^{(J)}\to\infty$
\item [(A3)] Asymptotic fraction stability: $n_*^{(J)}/N_*^{(J)}\to\gamma\in (0,1]$.
\end{itemize}
All the limits are considered as $J\to\infty$. These assumptions imply that $p_j^{(J)}\to 0$, not a conventional asymptotic regime to approximate Multinomial distribution. Indeed, this is the most appropriate regime given that $N_j$'s are modeled as bounded functions of case characteristics (in section 6).

\section{Conditional PDF under Random Sampling with Replacement}
As mentioned earlier, the WR sampling scheme implies the following about the distribution of $\textbf{n}^{(J)}$:
\[
\textbf{n}^{(J)} \sim \text{Multinomial}(n_*^{(J)},\textbf{p}^{(J)})
\]
Our first result identifies finite-dimensional distributions of $\textbf{n}^{(J)}$. We will use $[J]$ to denote the set of integers $\{1, \dots , J \}$ and $\mathbb{N}$ the set of all integers. For any $S\subset [J]$, let $\textbf{n}_S^{(J)}=(n_j^{(J)},j\in S)$.

\begin{prop}
Under (A1-3), for any subset $S$ of integers, of fixed cardinality $|S|$, we have 
\[
\textbf{n}_S^{(J)} \overset{d}{\longrightarrow} \bigotimes_{j=1}^{|S|} \text{Poisson}(\lambda_j), \quad
    \text{as $J \to \infty$}
\]
where $\lambda_j=\gamma N_j$.
\end{prop}
\begin{proof}
Without loss of generality, let $S=\{1,\ldots,K\}$ for some fixed $K\leq J$. $n_j^{(J)}$ is a sum of independent Bernoulli variables denoted by $B_1^j,\ldots,B_{n_*^{(J)}}^j$, for any $j\in [J]$. For any draw, $k$, we have $\sum_{j=1}^{J}B_k^j=1$. We can write the joint characteristic function of $\textbf{n}_S^{(J)}$ and take its limit:
\[
\begin{array}{rcl}
\varphi_{S}^{(J)}(t_1,\ldots,t_{K})&=&\mathrm{E} \text{ exp}\left(\sum_{j=1}^{K} it_j(B_1^j+\cdots+B_{n_*^{(J)}}^j)\right) \\ 
&=&\left(\mathrm{E} \text{ exp } i(t_1B_1^1+\cdots+B_1^{K})\right)^n \\
&=&\left(p_1\cdot e^{it_1}+\cdots+p_{K} \cdot e^{it_{K}}+(1-p_1-\cdots-p_{K}) \right)^n \\
&=&\left(\cfrac{\lambda_1 (e^{it_1}-1)+\cdots+\lambda_{K} (e^{it_{K}}-1)}{n}+1\right)^n \\
&\xrightarrow{n\rightarrow \infty}&\text{exp }\left(\lambda_1 (e^{it_1}-1)+\cdots+ \lambda_{K} (e^{it_{K}}-1) \right)
\end{array}
\]
which is the joint c.f. of a vector of independent Poisson variables.
\end{proof}

Unfortunately, this theorem is only useful if we are willing to use a small portion of the data to estimate the count model. According to this proposition, only if $K$ is sufficiently smaller than $J$, Poisson distribution works as a close approximation for Multinomial. However, similar result holds if $K=J$.

\begin{theorem}
Under (A1-3),
\[
\textbf{n}^{(J)} \overset{d}{\longrightarrow} \bigotimes_{j=1}^{J} \text{Poisson}(\lambda_j), \quad \text{conditioned on } \Sigma^J_{j=1}n_j=n_*^{(J)}
\]
, $\text{as $J \to \infty$}$.
\end{theorem}
\begin{proof}
\[
\begin{array}{rcl}
\text{Pr}(\textbf{n}^{(J)}=\textbf{n})&=&\cfrac{n_*^{(J)}!}{n_1!\cdots n_J!}\times \left(\cfrac{N_1}{N}\right)^{n_1}\cdots \left(\cfrac{N_J}{N}\right)^{n_J} \\
&=&\cfrac{n_*^{(J)}!}{(\rfrac{n_*^{(J)}}{e})^{n_*^{(J)}}}\times e^{\frac{n_*^{(J)}N_1}{N}}\cfrac{(\frac{n_*^{(J)}N_1}{N})^{n_1}}{n_1!}\times \cdots \times e^{\frac{n_*^{(J)}N_J}{N}}\cfrac{(\frac{n_*^{(J)}N_J}{N})^{n_J}}{n_J!} = \\
&=&\sqrt{2\pi n_*^{(J)}}\;(1+O({n_*^{(J)}})^{-1}) \times e^{\lambda_1}\cfrac{\lambda_1^{n_1}}{n_1!} \times \cdots \times e^{\lambda_J}\cfrac{\lambda_J^{n_J}}{n_J!}
\end{array}
\]
I have used Stirling's estimation for $n_*^{(J)}!$ in the second equation. The sum of last line over all $\textbf{n}=(n_1,\ldots,n_J)$ with $\sum_{j=1}^{J}n_j=n_*^{(J)}$ equals one, therefore $\sqrt{2\pi n_*^{(J)}}\;(1+O({n_*^{(J)}})^{-1})=\text{Pr}(\sum_{j=1}^{J}\text{Poisson}(\lambda_j)=n_*^{(J)})^{-1}$.
\end{proof}

\section{Conditional PDF under Random Sampling without Replacement}
Assuming that the sampled count variable is generated by a random sampling procedure without replacement (WOR),
\[
\textbf{n}^{(J)} \sim \text{Multi. Hypergeometric}(n_*^{(J)}, N_*^{(J)}, N_1,\ldots,N_J)
\]
I will derive the asymptotic behavior of the Multivariate Hypergeometric random vector under the same limiting conditions introduced in the previous section.
\begin{prop}
Under (A1-3), for any subset $S$ of integers, of fixed cardinality $|S|$, we have 
\[
\textbf{n}_S^{(J)} \overset{d}{\longrightarrow} \bigotimes_{j=1}^{|S|} \text{Binomial}(N_j,\gamma), \quad
    \text{as $J \to \infty$}
\]
\end{prop}

\begin{proof}
Without loss of generality, let $S=\{1,\ldots,K\}$ for some fixed $K\leq J$. $\text{Pr}(\textbf{n}_S^{(J)}=\textbf{n})=\cfrac{\binom{N_1}{n_1}\cdots \binom{N_{K}}{n_{K}} \binom{N_*^{(J)}-N_1-\cdots-N_{K}}{n_*^{(J)}-n_1-\cdots-n_{K}}}{\binom{N_*^{(J)}}{n_*^{(J)}}}$ can be written as:
\[
\begin{array}{ll}
\binom{N_1}{n_1}\cdots \binom{N_{K}}{n_{K}} \times &
\frac{n_*^{(J)}!}{(n_*^{(J)}-n_1-\cdots-n_{K})!} \times
\frac{(N_*^{(J)}-n_*^{(J)})!}{(N_*^{(J)}-N_1-\cdots-N_{K}-n_*^{(J)}+n_1+\cdots+n_{K})!} \times 
 \frac{(N_*^{(J)}-N_1-\cdots-N_{K})!}{N_*^{(J)}!} =\\
\binom{N_1}{n_1}\cdots \binom{N_{K}}{n_{K}} \times &
\cfrac{\prod_{i=0}^{n_1-1} (n_*^{(J)}-i)\times \prod_{i=0}^{N_1-n_1-1} (N_*^{(J)}-n_*^{(J)}-i)}{\prod_{i=0}^{N_1-1} (N_*^{(J)}-i)} \times \\
&  \cfrac{\prod_{i=0}^{n_2-1} (n_*^{(J)}-n_1-i)\times \prod_{i=0}^{N_2-n_2-1} (N-n_*^{(J)}-N_1+n_1-i)}{\prod_{i=0}^{N_2-1} (N_*^{(J)}-N_1-i)} \times \cdots \times \\
&\frac{\prod_{i=0}^{n_{K}-1} (n_*^{(J)}-n_1-\cdots-n_{K-1}-i)\times \prod_{i=0}^{N_{K}-n_{K}-1} (N_*^{(J)}-n_*^{(J)}-N_1+n_1-\cdots-N_{K-1}+n_{K-1}-i)}{\prod_{i=0}^{N_{K}-1} (N_*^{(J)}-N_1-N_{K-1}-i)}
\end{array}
\]
It converges to $\binom{N_1}{n_1}\cdots \binom{N_{K}}{n_{K}} \times \gamma^{n_1}(1-\gamma)^{N_1-n_1}\times\cdots\times \gamma^{n_{K}}(1-\gamma)^{N_{K}-n_{K}}$, as $J \to \infty$. 
\end{proof}

Like section 2, similar result holds if $K=J$.

\begin{theorem}
Under (A1-3),
\[
\textbf{n}^{(J)} \overset{d}{\longrightarrow} \bigotimes_{j=1}^{J} \text{Binomial}(N_j,\gamma), \quad \text{conditioned on } \Sigma^J_{j=1}n_j=n_*^{(J)}
\]
, $\text{as $J \to \infty$}$.
\end{theorem}

\begin{proof}
\[
\begin{array}{l}
\text{Pr}(\textbf{n}^{(J)}=\textbf{n})=\cfrac{\binom{N_1}{n_1}\cdots \binom{N_{J}}{n_{J}}}{\binom{N_*^{(J)}}{n_*^{(J)}}}= 
\binom{N_1}{n_1}\cdots \binom{N_{J}}{n_{J}} \times \cfrac{n_*^{(J)}!(N_*^{(J)}-n_*^{(J)})!}{N_*^{(J)}!} =\\
\binom{N_1}{n_1}\cdots \binom{N_{J}}{n_{J}} \times \frac{\sqrt{2\pi n_*^{(J)}}\; (\frac{n_*^{(J)}}{e})^{n_*^{(J)}}(1+O({n_*^{(J)}})^{-1}) \sqrt{2\pi (N_*^{(J)}-n_*^{(J)})}\; (\frac{N_*^{(J)}-n_*^{(J)}}{e})^{N_*^{(J)}-n_*^{(J)}}(1+O({N_*^{(J)}-n_*^{(J)}})^{-1})}{\sqrt{2\pi N_*^{(J)}}\; (\frac{N_*^{(J)}}{e})^{N_*^{(J)}}(1+O({N_*^{(J)}})^{-1})} =\\
\binom{N_1}{n_1}\cdots \binom{N_{J}}{n_{J}} \times \sqrt{2\pi (1-\gamma)n_*^{(J)}}(1+O({n_*^{(J)}})^{-1})\times (\frac{n_*^{(J)}}{N_*^{(J)}})^{n_1}(\frac{N_*^{(J)}-n_*^{(J)}}{N_*^{(J)}})^{n_1}\times \cdots \times (\frac{n_*^{(J)}}{N_*^{(J)}})^{n_J}(\frac{N_*^{(J)}-n_*^{(J)}}{N_*^{(J)}})^{n_J} =\\
\sqrt{2\pi (1-\gamma)n_*^{(J)}}(1+O({n_*^{(J)}})^{-1})\times \binom{N_1}{n_1}\gamma^{n_1}(1-\gamma)^{n_1}\times \cdots \times \binom{N_{J}}{n_{J}}\gamma^{n_J}(1-\gamma)^{n_J}
\end{array}
\]
I have used Stirling's estimation for $n_*^{(J)}!, N_*^{(J)}! \text{, and } (N_*^{(J)}-n_*^{(J)})!$ in the second line.  The sum of last line over all $\textbf{n}=(n_1,\ldots,n_J)$ with $\sum_{j=1}^{J}n_j=n_*^{(J)}$ equals one, therefore $\sqrt{2\pi (1-\gamma)n_*^{(J)}}(1+O({n_*^{(J)}})^{-1})=\text{Pr}(\sum_{j=1}^{J}\text{Binomial}(N_j, \gamma)=n_*^{(J)})^{-1}$.

\end{proof}

\section{Poisson Regression}
Returning to the finite-sample notation (dropping the superscripts), suppose the count model of interest is described as follows:
\[
N_j \sim \text{Poisson}(X_j^\prime \beta),\quad j=1,\ldots,J
\]
The goal is to estimate unknown parameter(s) $\beta$. We are going to find a simple form for the likelihood function. Notice that in order to identify the vector of parameters $\beta$ from the sample of $\textbf{n}=(n_1,\ldots,n_J)$ it is essential to know $N_*$. Otherwise we would be indecisive about the scale of $\beta$.

I will continue, in separate sub-sections, with the two assumptions we made about sampling procedure.

\subsection{Likelihood Function under WR}
The (asymptotic) conditional distribution of $n_j$ given $Nj$, and that of $N_j$ given $X_j$ is as follows:
\[
\begin{array}{rcl}
\text{Pr}(n_j=t|N_j)&\approx&e^{-\gamma N_j}\cfrac{(\gamma N_j)^t}{t!} \\
\text{Pr}(N_j=r|X_j)&=&e^{-X_j^\prime \beta}\cfrac{(X_j^\prime \beta)^r}{r!} 
\end{array}
\]
For a moment take $\gamma=\frac{n_*}{N_*}$ as a constant (given $X_j$). Hence the conditional distribution of $n_j$ given $X_j$ is 
\[
\text{Pr}(n_j=t|X_j)\approx\cfrac{e^{-X_j^\prime \beta}}{t!}\gamma^t\cdot \sum_{r=0}^{\infty}r^t\cfrac{(e^{-\gamma}X_j^\prime \beta)^r}{r!}
\]
Let $f(x,t)=\sum_{r=0}^{\infty}r^t\frac{x^r}{r!}$. The conditional probability will be written as $\frac{e^{-X_j^\prime \beta}}{t!}\gamma^t\cdot f(e^{-\gamma}X_j^\prime \beta,t)$. The following proposition holds about  $f(x,t)$.
\begin{prop}
Denote $\{g_t(x)\}_{t=0}^{\infty}$ as a sequence of polynomials satisfying  the following recursive equations:
\[
\begin{array}{l}
g_0(x)=1\\
g_t(x)^\prime=\sum_{m=0}^{t-1}g_m(x)\binom{t}{m},\quad t\geq 1
\end{array}
\]
Then, $f(x,t)=g_t(x)e^x,\: t=0,1,2,\ldots$.
\end{prop}

\begin{proof}
Writing slightly differently we have $f(x,t)=\sum_{r=0}^{\infty}(r+1)^t\cfrac{x^{r+1}}{(r+1)!}$. Therefore:
\[
f^\prime(x,t)=\sum_{r=0}^{\infty}(r+1)^t\cfrac{x^r}{r!}=\sum_{r=0}^{\infty}\sum_{m=0}^{t}\binom{t}{m}r^m\cfrac{x^r}{r!}=\sum_{m=0}^{t}\binom{t}{m}f(x,m)
\]
So we have $\left(e^{-x}f(x,t)\right)^\prime=\sum_{m=0}^{t-1}\binom{t}{m}e^{-x}f(x,m)$.
\end{proof}

[incomplete, the asymptotic log likelihood function?]

\subsection{Likelihood Function under WOR}
The likelihood function is more tractable under WOR sampling. If $\gamma$ were a constant, the univariate conditional distribution of $n_j$ would be a Poisson with the parameter adjusted by sampling ratio: 

\[
\begin{array}{rcl}
\text{Pr}(n_j=t|X_j)&=& \sum_{r\geq t}^{\infty} \text{Pr}(n_j=t|N_j=r) \text{Pr}(N_j=r|X_j) \\
&=& \sum_{r\geq t}^{\infty}\binom{r}{t}\gamma^t(1-\gamma)^{r-t}e^{-X_j^\prime \beta}\cfrac{(X_j^\prime \beta)^r}{r!} \\
&=& e^{-X_j^\prime \beta}\cfrac{(\gamma X_j^\prime \beta)^t}{t!}\sum_{r\geq t}^{\infty} \cfrac{((1-\gamma)X_j^\prime \beta)^{r-t}}{(r-t)!} \\
&=& e^{-X_j^\prime \beta}\cfrac{(\gamma X_j^\prime \beta)^t}{t!}\; e^{(1-\gamma)X_j^\prime \beta} \\
&=& \cfrac{(\gamma X_j^\prime \beta)^t}{t!}\; e^{-\gamma X_j^\prime \beta}
\end{array}
\]
The actual asymptotic likelihood of $n_j$ demands more effort to derive.
\begin{lemma}
For a fixed positive integer $k$ and,
\[
\sum_{i=0}^{M}e^{-\epsilon_M(i+k)^2+\delta_M(i+k)}\;\frac{x^i}{i!}\longrightarrow e^x ,\quad \text{as $M\to\infty$ and $\epsilon_M\overset{+}{\to} 0$ and $\delta_M\to 0$}
\]
assuming that $\delta_M=O(\sqrt{\epsilon_M})$.
\end{lemma}
\begin{proof}
For a fixed $m<M$, sufficiently large, truncate the sum at $i=m$. The following holds for the body and the tail:
\[
\begin{array}{rcl}
\sum_{i=0}^{m}e^{-\epsilon_M(i+k)^2+\delta_M(i+k)}\;\frac{x^i}{i!}&=&e^{O(\text{max}\{\epsilon_M,\delta_M\})}\;\sum_{i=0}^{m}\frac{x^i}{i!} \\
\sum_{i=m+1}^{M}e^{-\epsilon_M(i+k)^2+\delta_M(i+k)}\;\frac{x^i}{i!}&<&e^{-\epsilon_M(m+k)^2+\delta_M(m+k)}\;\sum_{i=m+1}^{M}\frac{x^i}{i!}\quad\text{(why?)}
\end{array}
\]
In asymptotics, the tail converges to zero and the body become $\sum_{i=0}^{m}\frac{x^i}{i!}$. Therefore enlarging $m$ will proof the claim.
\end{proof}
\begin{prop}
The chance of observing $n_k$ subject in the sample for case $k$ and the total population to be $N_*$ is:
\[
(n_k, N_*) \overset{d}{\longrightarrow} \text{Poisson}(\gamma X_k^\prime \beta)\bigotimes \mathcal{N}(\sum_{j\not=k}X_j^\prime \beta,\sum_{j\not=k}X_j^\prime \beta),\quad \text{as $J\to\infty$}
\]
\end{prop}

\begin{proof}
Let's $\textbf{X}=(X_1,\ldots,X_J)$ be the matrix of characteristics,
\[
\begin{array}{rcl}
\text{Pr}(n_k=t_k,N_*,n_*|\textbf{X})&=&\sum_{r_k=t_k}^{N_*}\text{Pr}(n_k=t_k|N_*,n_*,N_k=r_k,\textbf{X})\cdot \text{Pr}(N_*,n_*|N_k=r_k,\textbf{X})\cdot \text{Pr}(N_k=r_k|\textbf{X}) \\
&=&\sum_{r_k=t_k}^{N_*}\binom{r_k}{t_k}\gamma^{t_k}(1-\gamma)^{r_k-t_k}\frac{1}{\sqrt{2\pi \sum_{j\not=k}X^\prime_j \beta}}\; e^{-\frac{(N_*-r_k-\sum_{j\not=k}X_j^\prime \beta)^2}{2\sum_{j\not=k}X_j^\prime \beta}} e^{-X^\prime_k \beta}\frac{(X^\prime_k \beta)^{r_k}}{r_k!} \\
&=&e^{-X^\prime_k \beta} \frac{(\gamma X^\prime_k \beta)^{t_k}}{t_k!}\times\frac{1}{\sqrt{2\pi \sum_{j\not=k}X^\prime_j \beta}}\; e^{-\frac{(N_*-\sum_{j\not=k}X_j^\prime \beta)^2}{2\sum_{j\not=k}X_j^\prime \beta}}  \\
&&\quad\quad\quad\quad\quad\quad\sum_{r_k=t_k}^{N_*}e^{\frac{-r^2_k}{2\sum_{j\not=k}X^\prime_j \beta}} \text{exp}\left(\frac{N_*-\sum_{j\not=k}X_j^\prime \beta}{2\sum_{j\not=k}X_j^\prime \beta}\right)^{r_k} \cfrac{((1-\gamma)X^\prime_k \beta)^{r_k-t_k}}{(r_k-t_k)!} \\
&\to&e^{-\gamma X_k^\prime \beta} \frac{(\gamma X_k^\prime \beta)^{t_k}}{t_k!} \times \frac{1}{\sqrt{2\pi \sum_{j\not=k}X^\prime_j \beta}}\; e^{-\frac{(N_*-\sum_{j\not=k}X_j^\prime \beta)^2}{2\sum_{j\not=k}X_j^\prime \beta}} %\times  (1+O(\sum_{j\not=k}X_j^\prime \beta)^{-1})
\end{array}
\] 
The first equation is law of total probability. In the second equation, I apply central limit theorem to the sum of independent poisson variables, $N_1,\ldots,N_J$ given $\textbf{X}$. The third one uses the identity $\frac{(N_*-r_k-\sum_{j\not=k}X_j^\prime \beta)^2}{2\sum_{j\not=k}X_j^\prime \beta}=\frac{(N_*-\sum_{j\not=k}X_j^\prime \beta)^2}{2\sum_{j\not=k}X_j^\prime \beta}-\frac{N_*-\sum_{j\not=k}X_j^\prime \beta}{\sum_{j\not=k}X_j^\prime \beta}\; r_k+\frac{r^2_k}{2\sum_{j\not=k}X^\prime_j \beta}$. In the fourth, I use the lemma 4.1 and the implication of CLT that  $\underset{J\to\infty}{\text{plim}}\left(\frac{N_*-\sum_{j\not=k}X_j^\prime \beta}{2\sum_{j\not=k}X_j^\prime \beta}\right)=\underset{J\to\infty}{\text{plim}}\left(\frac{N_*-\sum_{j=1}^J X_j^\prime \beta}{2\sum_{j=1}^J X_j^\prime \beta}\right)=0$ with the rate of convergence equal $\frac{1}{\sqrt{\sum_{j=1}^J X_j^\prime \beta}}$.
\end{proof}
[incomplete, the full asymptotic log likelihood function?]

\section{Discussion}
It is common to predict a count variable with an ordered probit functional form:
\[
N_j>t \quad \text{if } X_j^\prime \beta+\epsilon_j > c_t\quad \text{for } t=0,1,\ldots
\]
where $\epsilon_j$ is a standard normal. $c_t$'s are parameters to be estimated, often normalized to have $c_0=0$. 

[incomplete]
\end{document}